\documentclass[a4paper,11pt,reqno]{amsart} 

\usepackage{amssymb,latexsym}
\usepackage{amsmath,amsthm}
\usepackage{enumerate}
\usepackage[mathscr]{eucal}
\usepackage{subcaption}
\usepackage{wrapfig}
\usepackage{graphicx}
\usepackage{stmaryrd}
\usepackage{url}

\usepackage{graphics} 

\theoremstyle{plain}
\newtheorem{theorem}{\indent Theorem}[section]

\newtheorem{lemma}[theorem]{\indent Lemma}
\theoremstyle{definition}
\newtheorem{definition}[theorem]{\indent Definition}
\newtheorem{example}[theorem]{\indent Example}
\theoremstyle{remark}

\allowdisplaybreaks

\newcommand{\disc}{\textup{\textbf{Disc}\thinspace}}

\begin{document}
\null
\vskip 1.8truecm

\title[Negotiation sets: a general framework]
{Negotiation sets: a general framework} 
\author{Tomasz Witczak}    
\address{Institute of Mathematics\\ Faculty of Science and Technology
\\ University of Silesia\\ Bankowa~14\\ 40-007 Katowice\\ Poland}
\email{tm.witczak@gmail.com} 


\begin{abstract}It is well-known fact that there exists $1-1$ correspondence between so-called double (or flou) sets and intuitionistic sets (also known as orthopairs). At first glance, these two concepts seem to be irreconcilable. However, one must remember that algebraic operations  in these two classes are also defined differently. Hence, the expected compatibility is possible. Contrary to this approach, we combine standard definition of double set with operations which are typical for intuitionistic sets. We show certain advantages and limitations of this viewpoint. Moreover, we suggest an interpretation of our sets and operations in terms of logic, data clustering and multi-criteria decision making. As a result, we obtain a structure of discussion between several participants who propose their "necessary" and "allowable" requirements or propositions.
\end{abstract}

\subjclass{Primary: 03B45, 03E75 ; Secondary: 68T27, 68T30}
\keywords{Intuitionistic sets, double sets, formal concept analysis, data clustering}


\maketitle

\section{Introduction}

Intuitionistic sets have been introduced by \c{C}oker (see \cite{3}). In fact, they are special (crisp) type of intuitionistic fuzzy sets (investigated earlier by Atanassov). Analogously, intuitionistic topological spaces (introduced by \c{C}oker in \cite{4}) can be considered as a crisp case of intuitionistic fuzzy topological spaces. Many authors contributed to the development of intuitionistic topologies since the beginning of XXI century. Nowadays, there are intuitionistic analogues of the majority of well-known topological notions (like interior, closure, density, continuity, compactness etc.). However, in this article we do not deal too much with topology. Rather, we are interested in sets as such. 

One thing should be pointed out: we do not participate in a terminological debate about whether it is sensible to call these sets (be they fuzzy or not) "intuitionistic". This topic has been extensively discussed by Guti\'{e}rrez Garcia and Rodabaugh in \cite{5}, Cattaneo and Ciucci in \cite{1} and also by Ciucci in \cite{2}. In general, we may agree with these authors that the term in question is not adequate. However, it seems that the notion of \emph{intuitionistic set} has been widely accepted or at least tolerated by many mathematicians.

Be as it may, the idea of intuitionistic set is simple: if we have non-empty universe $X$ with two subsets $A^1$ and $A^2$ such that $A^1 \cap A^2 = \emptyset$, then we define \emph{intuitionistic set} $A$ as $[A^1, A^2]$, assuming that if $A$ and $B$ are both intuitionistic sets, then $A \cap B = [A^1 \cap B^1, A^2 \cup B^2]$ and $A \cup B = [A^1 \cup B^1, A^2 \cap B^2]$. These notions are properly defined because the result of both operations is an intuitionistic set too. On the other hand, \emph{double sets} (see \cite{6}) are defined as ordered pairs $[A^1, A^2]$ such that $A^1 \subseteq A^2$. However, we assume that if $A$ and $B$ are double, then $A \cap B = [A^1 \cap B^1, A^2 \cap B^2]$ and $A \cup B = [A^1 \cup B^1, A^2 \cup B^2]$. Again, these operations do not lead out of the class. Now it is simple to show that each double set $[A^1, A^2]$ can be considered as an intuitionistic set $[A^1, -A^2]$, where $-A^2$ is just a complement of $-A^2$ with respect to the universe $X$. 

But there the question arises: is it reasonable to combine intuitionistic definitions of intersection and union with double sets? First of all, let us explain our motivation. Imagine that $A^1$ and $A^2$ are two sets of arguments, formulas, scenarios or certain other objects. We do not precise their nature: they can be items, scenarios or logical formulas. If these objects are only \emph{possible} (or \emph{allowable} a.k.a. \emph{admissible}), then they are somewhere in $A^2$. If they are \emph{necessary} (or \emph{obligatory}), then we put them into $A^1$. Hence, it as natural that $A^1$ is contained in $A^2$ (because what is necessary should also be possible; this is natural assumption, well-known from almost all modal logics which are not extremely weak). Assume now that we have two double sets $A$ and $B$. Let us consider the following new objects (or, equivalently, operations):

$[A^1 \cap B^1, A^2 \cup B^2]$

$[ (A^1 \cup B^1) \cap (A^2 \cap B^2), A^2 \cap B^2 ]$

These operations make sense in the light of our interpretation. Let us treat different double sets as different points of view of certain agents (or decision makers). Then our new objects describe two variants of \emph{compromise} between two agents, represented\footnote{Later we shall identify agents with their double sets.} by $A$ and $B$. In the first case, \emph{more} things are acceptable now and \emph{fewer} things are necessary. The second case is nearly symmetrical: more things are necessary and fewer things are acceptable. However, we must ensure ourselves that our new range of necessity is contained in the range of admissibility of both agents. For this reason, this kind of compromise is slightly more complicated than the first one (we intersect $A^1 \cup B^1$ with the intersection of $A^2$ and $B^2$). From the mathematical point of view, it means that our new object should still be a double set.

Our motivation is practical. For this reason, we are interested mostly in finite universes (collections of objects), even if it is possible to consider arbitrary ones. Moreover, we admit the thought that some versions of our framework may not solve each conflict but only some of them. This is typical for many decision systems: that some cases are left as unsolvable and some scenarios are equally good. In such situations we must use other, external criteria. These topics will be discussed in the last section of our paper.

\section{Formal exposition}

In this section we shall see formal (and more general) definitions of the functions and objects introduced earlier. Also, their basic algebraic properties will be studied. 

\subsection{Basic definitions}
The first three definitions are typical for double sets (see \cite{6} for example). However, we shall use the notion of \emph{negotiation set} to speak about double sets considered in the context of our new operations.

\begin{definition}
Assume that $X$ is a non-empty universe (consisting of \emph{objects}), $A^1, A^2 \subseteq X$ and $A^1 \subseteq A^2$. An ordered pair $[A^1, A^2]$ is called a \emph{negotiation set} (on $X$). We say that $A^1$ is \emph{range of necessity} (of $A$), while $A^2$ is its \emph{range of admissibility}. 
\end{definition}

\begin{definition}
Assume that $X$ is a non-empty universe and $A, B$ are two negotiation sets. We define \emph{complement} of $A$ as $-A = [-A^2, -A^1]$. Moreover, we define \emph{difference} of $A$ and $B$ as $A \setminus B = [A^1 \setminus B^2, A^2 \setminus B^1]$. We say that $A \subseteq B$ if and only if $A^1 \subseteq B^1$, $A^2 \subseteq B^2$. 

\end{definition}

\begin{definition}
Assume that $X$ is a non-empty universe, $J$ is an arbitrary set and $\{A_j: j \in J\}$ is a family of negotiation sets on $X$. We define the following operations:

\begin{enumerate}
\item \emph{Generalized union}:

$\bigcup_{j \in J} A_j = [\bigcup_{j \in J} A_j^1, \bigcup_{j \in J} A_j^2]$

\item \emph{Generalized intersection}:

$\bigcap_{j \in J} A_j = [\bigcap_{j \in J} A_j^1, \bigcap_{j \in J} A_j^2]$
\end{enumerate}
\end{definition}

It is easy to check that both complement and difference are well-defined: they result in a new negotiation set. The same can be said about union and intersection.

The next definition is new:

\begin{definition}
Assume that $X$ is a non-empty universe, $J$ is an arbitrary set and $\{A_j: j \in J\}$ is an indexed family of negotiation sets on $X$. We define the following operations:

\begin{enumerate}
\item \emph{Generalized minimalization} (of necessities): 

$\odot_{j \in J} A_j = [\bigcap_{j \in J}A_{j}^{1}, \bigcup_{j \in J}A_{j}^{2}]$

\item \emph{Generalized relative maximalization} (of necessities): 

$\oplus_{i \in J} A_j = [\bigcup_{j \in J} A_{j}^{1} \cap \bigcap_{j \in J} A_{j}^{2}, \bigcap_{j \in J} A_{j}^{2}]$
\end{enumerate}
\end{definition}

Note that for $J = \{1, 2\}$ we obtain binary instances that were presented in the introduction. In fact, we shall deal mostly with them: \\

$A \odot B = [A^1 \cap B^1, A^2 \cup B^2]$, $A \oplus B = [ (A^1 \cup B^1) \cap (A^2 \cap B^2), A^2 \cap B^2 ]$. \\

We can prove the following lemma:

\begin{lemma}
Let $X$ be a non-empty universe. Assume that $\{A_j; j \in J\}$ is a family of negotiation sets and $B$ is also negotiation set:

\begin{enumerate}
\item If $A_j \subseteq B$ for any $j \in J$, then $\odot_{j \in J} A_{j} \subseteq \bigcup_{j \in J} A_{j} \subseteq B$.

\item If $B \subseteq A_{j}$ for any $j \in J$, then $B \subseteq \bigcap_{j \in J} A_{j} \subseteq \oplus_{j \in J} A_{j}$. 

\end{enumerate}

\end{lemma}

\begin{proof}
\begin{enumerate}
\item Assume that for each $j \in J$ we have $A_j^1 \subseteq B^1$ and $A_j^2 \subseteq B^2$. Then it is clear that $\bigcap_{j \in J} A_j^1 \subseteq B^1$ and $\bigcup_{j \in J} A_{j}^2 \subseteq B^2$. Of course we can say more: that $\bigcap_{j \in J} A_j^1 \subseteq \bigcup_{j \in J} A_j^1 \subseteq B^1$. Thus we obtain our expected conclusion.

\item For any $j \in J$ we have $B^1 \subseteq A_j^1$ and $B^2 \subseteq A_j^2$. If $x \in B^1$, then $x \in \bigcap_{j \in J} A_j^1 \subseteq \bigcup_{j \in J} A_j^1$. However, for any $j \in J$, $A_j^1 \subseteq A_j^2$. Hence $x \in \bigcap_{j \in J} A_j^1 \subseteq \bigcup_{i \in J} A_j^1 \cap \bigcap_{j \in J}A_{j}^2$. If $x \in B^2$, then $x \in A_j^2$ for any $j \in J$. Thus $x \in \bigcap_{j \in J} A_j^2$ and we obtain our final conclusion.

\end{enumerate}

\end{proof}

Let us introduce another, weaker type of inclusion:

\begin{definition}
Let $X$ be a non-empty universe and assume that $A, B$ are two negotiation sets on $X$. We say that $A \subseteq^1 B$ if and only if $A^1 \subseteq B^1$. Analogously, $A \subseteq^2 B$ if and only if $A^2 \subseteq B^2$.  
\end{definition}

Now we may formulate the following theorem:

\begin{theorem}
\label{multi}

Let $X$ be a non-empty universe. Assume that $\{A_j; j \in J\}$ is a family of negotiation sets. Then we have:

\begin{enumerate}

\item $-\odot_{j \in J} A_j \subseteq^1 \oplus_{j \in J} (-A_j)$.

\item $\oplus_{j \in J} (-A_j) \subseteq^2 -\odot_{j \in J} A_j$.

\item $\odot_{j \in J} (-A_j) \subseteq^1 -\oplus_{j \in J}A_j$.

\item $-\oplus_{j \in J}A_j \subseteq^2 \odot_{j \in J} (-A_j)$. 

\end{enumerate}

\end{theorem}

\begin{proof}

First, let us write: 

$-\odot_{j \in J} A_j = [-\bigcup_{j \in J} A_j^2, -\bigcap_{j \in J} A_j^1]$, 

$\oplus_{j \in J} (-A_j) = [\bigcup_{j \in J} (-A_j^2) \cap \bigcap_{j \in J} (-A_j^1), \bigcap_{j \in J} (-A_j^1)]$. 

Second, we have:

$-\oplus_{j \in J} A_j = [-\bigcap_{j \in J} A_j^2, -\left( \bigcup_{j \in J} A_j^1 \cap \bigcap_{j \in J} A_j^2\right)]$,

$\odot_{j \in J} (-A_j) = [-\bigcup_{j \in J} A_j^2, -\bigcap_{j \in J} A_j^1]$.

Now we can prove all the cases:

\begin{enumerate}

\item Assume that $x \in -\bigcup_{j \in J} A_j^2$. This means that for any $j \in J$, $x \notin A_j^2$. Now it is easy to show that $x \in \bigcup_{j \in J} (-A_j^2)$. Moreover, for any $j \in J$ we have $A_j^1 \subseteq A_j^2$. Hence for any $j \in J$, $x \notin A_j^1$. Thus $x \in \bigcap_{i \in J} (-A_j^1)$. 

\item It follows from the general principles of set theory that $\bigcap_{j \in J} (-A_j^1) \subseteq -\bigcap_{j \in J} A_j^1$. 

\item Assume that $x \in -\bigcup_{j \in J}A_j^2$. Again, from the basic principles $x \in -\bigcap_{j \in J} A_j^2$. 

\item Let $x \in -\left( \bigcup_{j \in J} A_j^1 \cap \bigcap_{j \in J} A_j^2\right)$. Hence $x \notin \left( \bigcup_{j \in J} A_j^1 \cap \bigcap_{j \in J} A_j^2\right)$. It means that $x \notin \bigcup_{j \in J} A_j^1$ or $x \notin \bigcap_{j \in J} A_j^2$. Consider the first option. In particular, it means that $x \notin \bigcap_{j \in J} A_j^1$, i.e. $x \in -\bigcap_{j \in J} A_j^1$. As for the second one: there must be $k \in J$ such that $x \notin A_k^2$. Then $x \notin A_k^1$. Again, $x \notin \bigcap_{j \in J}A_j^1$.
\end{enumerate}

\end{proof}

We cannot replace weaker inclusions by equalities. Consider the following counter-example for Th. \ref{multi} 1, 2. 

\begin{example}

Assume that $X = \{a, b, c, d, e, f, g\}$, $A = [\{a, b\}, \{a, b, c, d\}]$ and $B = [\{c, d\}, \{c, d, g\}]$. Now: $A \odot B = [\emptyset, \{a, b, c, d, g\}]$. Then $-(A \odot B) = [\{e, f\}, \{a, b, c, d, e, f, g\}]$. On the other hand $-A = [\{e, f, g\}, \{c, d, e, f, g\}]$, $-B = [\{a, b, e, f\}, \{a, b, e, f, g\}]$. Then $-A \oplus -B = [\{e, f, g\}, \{e, f, g\}] \nsubseteq^1 -(A \odot B)$. Moreover, $-(A \odot B) \nsubseteq^2 -A \oplus -B$. 

\end{example}

In practical applications we limit our early attention to the collection of some distinguished negotiation sets: the \emph{initial ones} (just like $A$ and $B$ in the preceding example). They represent points of view of several discussants. Actually, this collection can be (and probably \emph{will be}) smaller than the class of all imaginable negotiation sets. On the other hand, negotiating sets resulting from the application of $\odot$ or $\oplus$ can be considered as \emph{coalitions of agents}.

\begin{example}

Adam ($A$), Bernard ($B$) and Clara ($C$) are planning joint trip by car. They are discussing which items should be taken to the trunk. They want to determine which things are necessary and which are only optional. Let us say that $X = \{a = \text{map}, b = \text{flashlight}, c = \text{shoe polish}, d = \text{first aid kit}, e = \text{fishing rod}, f = \text{ball}, g = \text{night-vision device}, h = \text{tent}, i = \text{riffle}, k = \text{guitar}, l = \text{chest expander}\}$. 

Assume that their negotiation sets are:

$A = [\{a, d\}, \{a, d, f, g, h\}]$

$B = [\{a, b, d\}, \{a, b, d, f, i, l\}]$

$C = [\{a, h\}, \{a, d, h, k\}]$. 

Suppose that they are interested in minimalization of necessities. Then\footnote{In the next section we shall prove that both $\odot$ and $\oplus$ are associative.}:

$(A \odot B) \odot C = [\{a, d\}, \{a, b, d, f, g, h, i, l\}] \odot C = [\{a\}, \{a, b, d, f, g, h, i, k, l\}]$. 

$(A \oplus B) \oplus C = [\{a, d\}, \{a, d, f\}] \oplus C = [\{a\}, \{a, d\}]$.

We may imagine that at the first step Bernard and Clara use, say, $\oplus$ to find their compromise, and then they use $\odot$ to discuss their proposition with Adam. Hence they obtain:

$(B \oplus C) \odot A = [\{a\}, \{a, d\}] \odot A = [\{a\}, \{a, d, f, g, h\}]$. 

In each of these three cases they are only sure that a map should be taken. However, they have obtained three different ranges of admissibility. Of course their debate may take longer. 
\end{example}

The next example is more abstract and rather theoretical:

\begin{example}
Let $X = \mathbb{C}$ (the set of complex numbers). Assume that $A = [\mathbb{N}, \mathbb{R}]$, $B = [\{1, 2, 3\}, \mathbb{N}]$ and $C = [\mathbb{Z}, \mathbb{C}]$. We may calculate (among many other possible compositions): 

$A \odot B = [\{1, 2, 3\}, \mathbb{R}]$, $B \oplus C = [\mathbb{N}, \mathbb{N}]$, $(A \oplus B) \odot C = [\mathbb{N}, \mathbb{N}] \odot [\mathbb{Z}, \mathbb{C}] = [\mathbb{N}, \mathbb{C}]$.

\end{example}

\subsection{Algebraic properties}
Undoubtedly, some fundamental algebraic properties of our system should be checked. This is essentially a content of Th. \ref{yesprop} and Th. \ref{noprop}. 

\begin{theorem}
\label{yesprop}

Let $X \neq \emptyset$, $J$ be an arbitrary set and $\{A_j: j \in J\}$ be a family of negotiation sets on $X$. Then the following features of binary $\odot$ and $\oplus$ operations \textbf{hold}:

\begin{enumerate}

\item \textbf{Idempotence}

\begin{proof} Clearly, $A \oplus A = A \odot A = A$. \end{proof}

\item \textbf{Commutativity}.

\begin{proof}This is trivial and based on the commutativity of $\cup$ and $\cap$ operations. \end{proof}

\item $\oplus \odot$-\textbf{Absorption law}.

\begin{proof}

We may write:

$[A^1, A^2] \oplus [A^1 \cap B^1, A^2 \cup B^2] = [(A^1 \cup (A^1 \cap B^1)) \cap (A^2 \cap (A^2 \cup B^2)), A^2 \cap (A^2 \cup B^2)] = [A^1 \cap A^2, A^2] = [A^1, A^2]$. 

\end{proof}

\item \textbf{Associativity}.

\begin{proof}
First, consider $\odot$. We have: 

$A \odot (B \odot C) = A \odot [B^1 \cap C^1, B^2 \cup C^2] = [A^1 \cap B^1 \cap C^1, A^2 \cup B^2 \cup C^2] = (A \odot B) \odot C$. 

Second, we write:

$A \oplus (B \oplus C) = [A^1, A^2] \oplus [(B^1 \cup C^1) \cap (B^2 \cap C^2), B^2 \cap C^2] = [(A^1 \cup ((B^1 \cup C^1) \cap (B^2 \cap C^2))) \cap (A^2 \cap B^2 \cap C^2), A^2 \cap B^2 \cap C^2] = [(A^1 \cup B^1 \cup C^1) \cap (A^1 \cup (B^2 \cap C^2)) \cap (A^2 \cap B^2 \cap C^2), A^2 \cap B^2 \cap C^2]$. 

and, simultaneously:

$(A \oplus B) \oplus C = [(A^1 \cup B^1) \cap (A^2 \cap B^2), A^2 \cap B^2] \oplus (C^1, C^2) = [(((A^1 \cup B^1) \cap (A^2 \cap B^2)) \cup C^1) \cap (A^2 \cap B^2 \cap C^2), A^2 \cap B^2 \cap C^2] = [(A^1 \cup B^1 \cup C^1) \cap ((A^2 \cap B^2) \cup C^1)) \cap (A^2 \cap B^2 \cap C^2), A^2 \cap B^2 \cap C^2]$. \\

We see that ranges of admissibility of resulting sets are identical. What about their ranges of necessity? 

\textbf{i)} Let $x \in (A^1 \cup B^1 \cup C^1) \cap (A^1 \cup (B^2 \cap C^2)) \cap (A^2 \cap B^2 \cap C^2)$. In particular, it means that $x \in A^2 \cap B^2 \cap C^2$. Assume that $x \notin (A^1 \cup B^1 \cup C^1) \cap ((A^2 \cap B^2) \cup C^1)) \cap (A^2 \cap B^2 \cap C^2)$. The only option worth considering is $x \notin (A^2 \cap B^2) \cup C^1$. However, in this case $x \notin A^2 \cap B^2$, and this is contradiction.

\textbf{ii)} Let $x \in (A^1 \cup B^1 \cup C^1) \cap ((A^2 \cap B^2) \cup C^1)) \cap (A^2 \cap B^2 \cap C^2)$. In particular, it means that $x \in A^2 \cap B^2 \cap C^2$. Assume that $x \notin (A^1 \cup B^1 \cup C^1) \cap (A^1 \cup (B^2 \cap C^2)) \cap (A^2 \cap B^2 \cap C^2)$. The only option worth considering is $x \notin A^1 \cup (B^2 \cap C^2)$. However, in this case $x \notin A^1$ and $x \notin B^2 \cap C^2$. This is contradiction. 

\end{proof}

\end{enumerate}

\end{theorem}

\begin{theorem}
\label{noprop}

Let $X \neq \emptyset$, $J$ be an arbitrary set and $\{A_j: j \in J\}$ be a family of negotiation sets on $X$. Then the following features of binary $\odot$ and $\oplus$ operations  \textbf{do not} hold:

\begin{enumerate}

\item \textbf{Distributivity}.

\begin{proof}
Counterexamples: 

\textbf{1)} First, consider $A, B, C$ where $A^1 = \{x\}, A^2 = \{a, x\}$, $B^1 = \{b\}, B^2 = \{b, d\}$, $C^1 = \{c\}, C^2 = \{c, x\}$. Now:

$B \odot C = [\{b\} \cap \{c\}, \{b, d\} \cup \{c, x\}] = [\emptyset, \{b, c, d, x\}]$,

$A \oplus (B \odot C) = [\{x\} \cap \{x\}, \{x\}] = [\{x\}, \{x\}]$,

$A \oplus B = [\{x, b\} \cap \emptyset, \emptyset] = [\emptyset, \emptyset]$,

$A \oplus C = [\{x, c\} \cap \{x\}, \{x\}] = [\{x\}, \{x\}]$.

Finally, we obtain:

$(A \oplus B) \odot (A \oplus C) = [\emptyset \cap \{x\}, \emptyset \cup \{x\}] = [\emptyset, \{x\}] \neq A \oplus (B \odot C)$. 

\textbf{2)} Second, consider $A, B, C$ where $A^1 = \{a\}, A^2 = \{a\}, B^1 = \{x\}, B^2 = \{x, b\}, C^1 = \{x, a\}, C^2 = \{x, a, c\}$. Now:

$A \odot B = [\{a\} \cap \{x\}, \{a\} \cup \{x, b\}] = [\emptyset, \{a, x, b\}]$,

$A \odot C = [\{a\} \cap \{x, a\}, \{a\} \cup \{x, a, c\}] = [\{a\}, \{x, a, c\}]$,

$(A \odot B) \oplus (A \odot C) = [(\emptyset \cup \{a\}) \cap \{a, x\}, \{a, x\}] = [\{a\}, \{a, x\}]$,

$B \oplus C = [(\{x\} \cup \{x, a\}) \cap \{x\}, \{x\}] = [\{x\}, \{x\}]$.

Finally, we obtain:

$A \odot (B \oplus C) = [\{a\} \cap \{x\}, \{a\} \cup \{x\}] = [\emptyset, \{a, x\}] \neq (A \odot B) \oplus (B \odot C)$
\end{proof}

\item $\odot \oplus$-\textbf{Absorption law}. 

\begin{proof}

Counterexample:

Consider $A, B$ where $A^1 = \{x\}, A^2 = \{x\}, B^1 = \{b\}, B^2 = \{b\}$. Now:

$A \oplus B = [(\{x\} \cup \{b\}) \cap (\{x\} \cap \{b\}), \{x\} \cap \{b\}] = [\emptyset, \emptyset]$.

Hence we get the following result:

$A \odot (A \oplus B) = [\{x\} \cap \emptyset, \{x\} \cup \emptyset] = [\emptyset, \{x\}] \neq [\{x\}, \{x\}] = A$. 

\end{proof} 

\end{enumerate}

\end{theorem}

The two theorems above allow us to say that the set of all negotiation sets on $X$ forms a structure which is slightly weaker than lattice (and slightly stronger than semi-lattice). As we could see, it satisfies only one absorption law. 

In this environment we can speak about analogues of empty set and the whole universe:

\begin{definition}
Assume that $X$ is a non-empty universe and $x \in X$. We point out two special types of negotiation sets: $\emptyset_{N} = [\emptyset, \emptyset]$ and $X_{N} = [X, X]$. 
\end{definition}

\begin{lemma}
\label{zeroone}
Assume that $X$ is a non-empty universe and $A$ is negotiation set on $X$. Then we have:

\begin{enumerate}
\item $A \odot X_{N} = [A^1, X]$.

\item $A \odot \emptyset_{N} = [\emptyset, A^2]$.

\item $A \oplus X_{N} = [A^2, A^2]$.

\item $A \oplus \emptyset_{N} = \emptyset_{N}$. 

\end{enumerate}
\end{lemma}

As we can see, there is no full analogy between $\emptyset_{N}$ (resp. $X_{N}$) and lattice bottom $\mathbf{0}$ (resp. top $\mathbf{1}$). We may also check \emph{half-empty} set $X_{P} = [\emptyset, X]$. 

\begin{lemma}
\label{one}
Assume that $X$ is a non-empty universe and $A$ is negotiation set on $X$. Then we have:

\begin{enumerate}
\item $A \odot X_{P} = X_{P}$.

\item $A \oplus X_{P} = A$.  
\end{enumerate}
\end{lemma}

Here we see that $X_{P}$ plays the role of $\mathbf{1}$: if $\odot$ and $\oplus$ are treated as (resp.) $\lor$ and $\land$. Note that this identification allows us to say that $\emptyset_{N}$ behaves in a way like $\mathbf{0}$. However, this correspondence is only partial: note that $A \oplus \emptyset_{N}$ does not result in $A$. 

\emph{Per analogiam} with intuitionistic points, we may speak about \emph{negotiation points}:

\begin{definition}
Assume that $X$ is a non-empty universe and $x \in X$. We point out two special types of negotiation sets: $x_{0.5} = [\emptyset, \{x\}]$ and $x_{1} = [\{x\}, \{x\}]$. 
\end{definition}

\begin{lemma}
\label{points}
Let $x, y \in X$ and $x \neq y$. Then we have the following results:

\begin{enumerate}
\item $x_{0.5} \oplus y_{0.5} = x_{0.5} \oplus y_{1} = x_{1} \oplus y_{1} = [\emptyset, \emptyset]$.
\item $x_{0.5} \odot y_{0.5} = x_{0.5} \odot y_{1} = x_{1} \odot y_{1} =  [\emptyset, \{x, y\}]$. 
\end{enumerate}

\end{lemma}

In the next section we shall deal not only with our negotiation sets but also with the internal structure of universe $X$. 

\section{Contradictions and relationships}

We assumed, by default, that our basic objects (the elements of $X$) are in some sense mutually independent and peacefully coexisting, while our points of view (negotiation sets) are equal: there is no any specified hierarchy between them. However, these assumptions do not exhaust the whole richness of reality. First of all, we may easily imagine that there is a kind of conflict between some of our objects. In short, they can be mutually contradictory. For example, if our objects are logic formulas (sentences), then it is enough to consider pair $x = \varphi$ and $y = \lnot \varphi$ (we assume tacitly that our logic is not paraconsistent). It is also possible that $x$ and $y$ are two scenarios which are (in practice) mutually exclusive. For example, $x$ means "holidays\footnote{In the sense of vacations.} on Tristan da Cunha", while $y$ means "holidays on St Kilda" (at the same time and with the same participants). Clearly, it is not possible to be in two different locations at once. Moreover, it is possible that (due to the certain reasons like lack of cash) we cannot visit these two places even one after another. Hence, these scenarios are still incompatible.

Note that these contradictions are inherent for the elements in question: they do not depend on the points of view of the agents. Of course, one could discuss another approach: that $x$ and $y$ are contradictory from the perspective of agent $A$ but agent $B$ does not recognize any conflict between them. Undoubtedly, this model should also be investigated but it is beyond the scope of present study. 

In this section we shall introduce two variants of contradiction: stronger and weaker one. The idea is that our negotiation sets should be \emph{consistent} (in some sense which will be defined later). If a given set is consistent in this sense, then we say that it is admitted to discussion: it belongs to the class \disc. 

Let us introduce some basic definitions:

\begin{definition}
Assume that $X$ is a non-empty universe and $\lightning$, $\wr$ are two binary, symmetric and anti-reflexive relations on $X$. We define \disc as the following class of negotiation sets: $A \in \disc$ if and only if the two requirements below are satisfied:

\begin{enumerate}
\item There are no any $x$, $y$ in $A^2$ such that $x \lightning y$. 

\item There are no any $x$, $y$ in $A^2$ such that $x \wr y$ and ($x \in A^1$ or $y \in A^1$).

\end{enumerate}

\end{definition}

This definition says that $\lightning$ reflects the idea of strong contradiction: it is not possible that two strongly contradictory objects are in the range of admissibility of a consistent negotiation set. It is possible that two weakly contradictory objects (i.e. such that $x \wr y$) are in this range. However, they both should be in $A^2 \setminus A^1$. It means that we leave them for a further discussion (so they are temporarily admissible) but cannot assume that even one of them is necessary.

We may prove the following theorem:

\begin{theorem}
Let $X$ be a non-empty universe with relations $\lightning$ and $\wr$. Then the set \disc is closed under operation $\oplus$. 
\end{theorem}

\begin{proof}

Assume that $A, B \in \disc$ and suppose that our thesis is not true: $(A \oplus B) \notin \disc$. There are two possible reasons:

\begin{enumerate}
\item There are $x, y \in (A \oplus B)^2$ such that $x \lightning y$. However, $(A \oplus B)^2 = A^2 \cap B^2$, hence $x, y \in A^2$ and $x, y \in B^2$. This is contradiction, because $A, B \in \disc$. 

\item There are $x, y \in (A \oplus B)^2$ such that $x \wr y$ and (without loss of generality) $x \in (A \oplus B)^1$, $y \in (A \oplus B)^2$. Hence $x \in (A^1 \cup B^1) \cap (A^2 \cap B^2)$. In particular, it means that $x \in A^1$ or $x \in B^1$. However, at the same time $y \in A^2$ and $y \in B^2$. This gives us contradiction because $x \wr y$ and $A, B \in \disc$, hence it is not possible that $x \in A^1, y \in A^2$ or $x \in B^1, y \in B^2$. 

\end{enumerate}

\end{proof}

What about $\odot$ operation? We have this partial result:

\begin{lemma}

Let $X$ be a non-empty universe with relation $\wr$ and $A, B \in \disc$. Then there are no such $x, y \in (A \odot B)^2$ that $x \wr y$ and (without loss of generality) $x \in (A \odot B)^1$, $y \in (A \odot B)^2$.

\end{lemma}

\begin{proof}
Assume the contrary. Hence, $x \in (A \odot B)^1$, $y \in (A \odot B)^2$. It means that $x \in A^1$ and $x \in B^1$. At the same time $y \in A^2$ or $y \in B^2$. Clearly, we obtain contradiction (because $x \wr y$).

\end{proof}

However, things are more complicated if we consider $\lightning$ relation. Let us discuss the following (counter)-example:

\begin{example}
\label{contrex}
Let $X$ be a non-empty universe with relation $\lightning$. Assume that $a \lightning b$ (in particular, it means also that $a \neq b$) and $A = [\{a\}, \{a\}]$, $B = [\{b\}, \{b\}]$. Then both $A$ and $B$ belong to \disc. Now consider $A \odot B = [\emptyset, \{a, b\}]$. Then $A \odot B \notin \disc$. 
\end{example}

Several solutions to this problem can be applied:

\begin{enumerate}
\item We may assume that only some actions with use of $\odot$ are \emph{permissible}. If there a problem occurs (similar to the one from Ex. \ref{contrex}), then we just say that it is not possible to solve it. This round of negotiations fails. 

\item We may assume that each two strongly contradictiory objects $x$, $y$ are also connected by asymmetric, transitive and anti-reflexive relation $>$. The interpretation goes as follows: if $A, B \in \disc$ but $A \odot B \notin \disc$ because of some $x, y \in (A \odot B)^2$ such that $x \lightning y$ and (without loss of generality) $x > y$, then we leave only $x$ in $(A \odot B)^2$. This solution has one advantage: each conflict can be solved. However, relation $>$ does not depend on the point of view: it is inherent for the pairs of objects themselves. While in the case of $\lightning$ this approach is natural, the same cannot be said about $>$. Clearly, it means that all participants of discussion agree that, for example, holidays on Mars are always better when compared to the holidays on the Moon. 

\item We may determine certain (total) order not between objects but between our (initial) negotiation sets. This would mean that our points of view are hierarchically ranked. In other words, if agent $A$ has priority over $B$ and $B$ has priority over $C$, then we may write: $A > B > C$. If $A, B \in \disc$ but $A \odot B \notin \disc$ because of some $x, y \in (A \odot B)^2$ such that $x \lightning y$ and $x \in A^2$, then we leave only $x$ in $(A \odot B)^2$. 

This solution has one serious flaw: it says nothing about the hierarchy of sets obtained by means of $\oplus$ and $\odot$. Basically, it refers only to the initial, "pure" points of view. For example, it is natural that the president's opinion is more important that the opinion of vice-president; and the opinion of the latter is stronger than the opinion of secretary of state\footnote{Clearly, this description is simplified and skips some nuances.}. However, if we assume arbitrarily that, say, $A > B > C$, then it does not give us any information about the relationship between $A \oplus B$ and $C$ or $A \odot C$ and $B \oplus C$. One could say that if we have finitely many finite initial sets, then we can determine hierarchy of each two possible results. However, this would be artificial and impractical.

Of course we may invent more precise relations. For example, we can assume that $A > B$ if $A^1$ contains \emph{fewer} elements than $B^1$. It would mean that priority goes to those agents (or coalitions of agents) which have lesser expectations regarding necessity. However, such relation is not total. Hence, it does not solve all contradictions.

\end{enumerate}

\section{Further studies}
Other problems should also be a matter of further studies. The fact that certain object is strongly contradictory with another one does not always mean that choosing one of them is satisfying. For example, $x = $ "holidays in a tropical country" is contradictory with $y = $ "holidays on Svalbard". However, one could say that these scenarios belong to different classes. Being on Svalbard is something concrete, while being "in a tropical country" is vague. In some sense, these are philosophical and linguistic considerations. On the other hand, such subtle distinctions can be expressed in a formal language of functions and relations.

Moreover, there are also algorithmic issues. The question is: which negotiation set should be considered as a final effect of discussion; and is it possible to determine steps leading to the achievement of this goal?

Another interesting topic is topology. It is not clear which of our operations is "closer" to the idea of union and which resembles intersection. Perhaps $\odot$ has more to do with union because it joins together external parts of negotiation sets, hence resulting sets are "bigger" when considered just as subsets of $X$. Hence, it may be valuable to consider families closed under arbitary minimalizations of necessities (that is, $\odot$) and finite maximalizations (that is, $\oplus$). However, we may imagine that the roles of these two operators are replaced.

\end{document}